\documentclass[preprint]{elsarticle}
\usepackage{array}
\usepackage{float}

\usepackage{amsmath} 
\usepackage{amssymb}  
\usepackage{theorem}

\def\D{\mathcal{D}}
\def\C{\mathcal{C}}

\def\e{\epsilon}
\def\d{\delta}

\def\a0{($\e$,\,0)}
\def\Prob{\mathbb{P}}
\def\R{\mathbb{R}}
\def\N{\mathbb{N}}

\newtheorem{lemma}{Lemma}
\newtheorem{theorem}{Theorem}
\newtheorem{corollary}{Corollary}
\newtheorem{definition}{Definition}

\newtheorem{example}{Example}

\DeclareMathOperator*{\diag}{diag}
\DeclareMathOperator{\rank}{rank}
\DeclareMathOperator{\trace}{tr}
\DeclareMathOperator{\conv}{conv}
\DeclareMathOperator{\ext}{ex}
\DeclareMathOperator{\spn}{span}

\begin{document}

\title{Extreme Points of the Local Differential Privacy Polytope}
\author[naoise_doug]{Naoise Holohan}
\author[naoise_doug]{Douglas J. Leith}
\author[ollie]{Oliver Mason\corref{cor1}}

\address[naoise_doug]{School of Computer Science and Statistics, Trinity College Dublin, Ireland}
\address[ollie]{Dept. of Mathematics and Statistics/Hamilton Institute, Maynooth University, Co. Kildare, Ireland \& Lero, the Irish Software Research Centre}
\cortext[cor1]{Corresponding author. Tel.: +353 (0)1 7083672; fax: +353
5(0)1 7083913; email: oliver.mason@nuim.ie}

\begin{abstract}
We study the convex polytope of $n\times n$ stochastic matrices that define locally $\epsilon$-differentially private mechanisms.  We first present invariance properties of the polytope and results reducing the number of constraints needed to define it.  Our main results concern the extreme points of the polytope.  In particular, we completely characterise these for matrices with 1, 2 or $n$ non-zero columns.  

\end{abstract}


\maketitle

\section{Introduction}

Data privacy has been of interest to researchers in computer science \cite{AW89}, statistics, cryptography \cite{DH79} and law \cite{budnitz1997privacy} for decades.  The recent emergence of `Big Data', while offering significant potential benefits to business and society, poses very real risks to personal privacy; this naturally has led to increased interest in questions pertaining to data privacy.  The concept of \emph{Differential Privacy}, introduced by C. Dwork in 2006 \cite{Dwo06}, has emerged as a popular theoretical paradigm in privacy research within the computer science community and has been applied to various different types of data and queries \cite{Dwo08}.

We are interested in the geometry of matrix polytopes arising in the study of differential privacy for categorical or finite-valued datasets.  More formally, we consider databases $\mathbf{d} \in D^N$ where the set $D$ is finite and can, without loss of generality, be taken to be $\{1, \ldots, n\}$.  Each entry in $\mathbf{d}$, $d_i$, corresponds to data contributed by an individual; the base set $D$ describes all the values that data entries can take.  

The problem we consider is motivated by the construction of differentially private sanitisations, where we are interested in releasing a private, sanitised version of a database $\mathbf{d}$.  A \emph{sanitisation} is defined by a set of random variables $X_{\mathbf{d}}$ taking values in $D^N$ for every $\mathbf{d} \in D$.  Loosely speaking, $X_{\mathbf{d}}$ describes a noisy version of the original $\mathbf{d}$ designed to protect the privacy of individual data contributors.  

The differential privacy model specifies two privacy parameters, $\e\ge0$ and $0\le\d\le1$. For any two databases $\mathbf{d},\mathbf{d}'\in D^N$ that differ in one row only, ($\e$,$\d$)-differential privacy requires
\begin{equation}
\Prob(X_{\mathbf{d}}\in A)\le e^\e\Prob(X_{\mathbf{d}'} \in A)+\d,
\end{equation}
for all $A\subseteq D^N$.

In essence, differential privacy ensures that answers to queries on a database cannot change greatly when one person's information in a database is altered.

The above definition considers global privacy with the mechanism defined on a complete database.  Global mechanisms can readily be constructed using locally private mechansism, where  subjects perturb/sanitise their own data locally before providing it to a central database upon which queries are answered \cite{DJW13}.  The concept of local privacy first appeared over 50 years ago as a way to eliminate bias in surveying \cite{War65} and is known in other contexts as \emph{input perturbation} or \emph{randomised response} \cite{War65, GKS08}.  A rigorous mathematical framework has been developed which guarantees global differential privacy when local differential privacy methods are applied \cite{HLM15}.



We refer to local mechanisms as 1-dimensional mechanisms, as they take a single row of a database as an input, and output another (perturbed/noisy) row.  In our context, a 1-dimensional mechanism is specified by giving an appropriate probability mass function $p_{i}$ for every $i \in D = \{1, \ldots, n\}$.  More compactly, a 1-dimensional mechanism is defined by a stochastic matrix $A \in \mathbb{R}^{n \times n} $ where $a_{ij}$ denotes the probability of outputting $j$ when the input, or real data, is $i$.  The requirement for local differential privacy is then given by:
\begin{equation}\label{eq:locdp}
a_{ik} \le e^\e a_{jk} + \d
\end{equation}
for all $i, j, k$.  These constraints, taken together with the stochastic and nonnegativity constraints, define the local differential privacy polytope.  We shall consider the simplified case of strict differential privacy (which is what was originally introduced by Dwork) where $\d = 0$ here.  

In practice, we are interested in finding a mechanism (i.e. a matrix in this polytope) which is optimal for some utility function.  Understanding the geometry of the polytope guides the design of such mechanisms.  For instance, if the utility function happens to be linear, then the optimal mechanism occurs at an extreme point of the polytope.  The search for optimal mechanisms in differential privacy has been studied by a number of authors \cite{MT07,NST12,LHR10}.  Local differential privacy has been studied recently in the paper \cite{DJW13}, while extremal local differential privacy mechanisms were considered in \cite{KOV14}.  Of course, polytopes of stochastic matrices and doubly stochastic matrices have been widely studied in the past \cite{BEK05, BL91,CR99,Bru85}.  An alternative study on geometrical aspects of differential privacy can be found in \cite{HT10}.





The basic layout of the paper is as follows.  In Section~\ref{sc:prelim} we introduce preliminary definitions of polytopes, extreme points, the concept of differential privacy and the polytope with which we will be working. In Section~\ref{sc:elem} we look at some elementary results for extreme points of this polytope, and in Section~\ref{sc:main} we present our main results. In Section~\ref{sc:special} we examine a number of special cases for extreme points, and finish with some concluding remarks in Section~\ref{sc:conc}.

\section{Notation and Background}\label{sc:prelim}

To begin, let us introduce the major notation and standard definitions to be used in our results.  For a matrix $A \in \mathbb{R}^{n \times n}$ and $1 \leq i \leq n$, we will use $A^{(i)}$ to denote the $i$th column of $A$.  $A^T$ denotes the usual matrix transpose.  We denote by $\mathbf{1}$ the (column) vector of all ones where the dimension will typically be clear from context.  We denote by $e_i$, $1\leq i \leq n$, the $i$th standard basis vector of $\mathbb{R}^n$.  

\subsection{Polyhedra}

In this paper, we adopt the following definitions for polyhedra and polytopes

\begin{definition}\label{df:polyt}
Let $\langle\cdot,\cdot\rangle:V\times V\to \mathbb{R}$ be an inner product on a real vector space $V$, and let $\left\{c^{(1)},\dots,c^{(q+l)}\right\}\subseteq V$ and $b\in \mathbb{R}^{q+l}$ be given. A convex polyhedron $\mathcal{P}\subseteq V$ is defined as:
\begin{equation}\label{eq:polyhedron}
\mathcal{P}=\left\{v\in V:\begin{array}{rl}\langle c^{(i)},v\rangle= b_i, & \forall\;1\le i\le q,\\
\langle c^{(q+i)},v\rangle\le b_{q+i}, & \forall\;1\le i\le l.\end{array}\right\}.
\end{equation}
\end{definition}

An inequality constraint is said to be \emph{tight} or \emph{active} on a point $v$ if $\langle c^{(q+i)},v\rangle = b_{q+i}$. 
\begin{definition}
A convex polytope in a vector space $V$ is the convex hull of a finite collection of points in $V$.
\begin{equation}\label{eq:polytope}
\mathcal{P}=\conv(v_1, \dots, v_k),
\end{equation}
where $v_i\in V$ for all $i$.
\end{definition}

It is well known that all polytopes are polyhedra, but only bounded polyhedra are polytopes.  

An \emph{extreme point} of a polyhedron cannot be written as the convex combination of any other points in the polyhedron.

\begin{definition}[Extreme point]
Let $\mathcal{P}\subseteq\R^n$ be a convex polyhedron. A point $v\in\mathcal{P}$ is an extreme point of $\mathcal{P}$ if $w,z\in\mathcal{P}$, $\frac{1}{2}(w+z)=v$, implies $w=z=v$.
\end{definition}

We denote by $\ext(\mathcal{P})$ the set of all extreme points of a polyhedron $\mathcal{P}$. 

Our primary interest is in characterising the extreme points of the local differential privacy polyhedron.  The following theorem from convex geometry shall prove useful in this regard \cite{barvinok2002course}.

\begin{theorem}\label{th:barvinok}
Let $\mathcal{P}\subseteq V$ be a polyhedron, and consider a point $v\in\mathcal{P}$. Denote by the set $I_v\subseteq\{1,\dots,l\}$ the indices of the inequality constraints that are tight on $v$ (i.e.\ $\langle c^{(q+i)},v\rangle =b_{q+i}$ for all $i\in I_v$ and $\langle c^{(q+i)},v\rangle<b_{q+i}$ for all $i\in\{1,\dots,l\}\setminus I_v$). Then $v\in\ext(\mathcal{P})$ if and only if
$$\spn\left(\left\{c^{(1)},\dots,c^{(q)}\right\}\cup\left\{c^{(q+i)}: i\in I_v\right\}\right)=V.$$
\end{theorem}
Essentially, this result tells us that $v$ is an extreme point of $P$ if and only if there are $n$ linearly independent constraints tight on $v$ where $n$ is the dimension of $V$.
\subsection{Differential Privacy}\label{sc:dp}

As in the Introduction, we take the set $D=\{1, \dots, n\}$ to be the domain of the rows of our database (i.e.\ each subject contributes a value of $D$ to the database).  For local differential privacy to be satisfied, we require:
$$\Prob(X_{i}\in I)\le e^\e \Prob(X_{j}\in I)+\d,$$
for all $i,j\in\{1, \dots, n\}$ and for all $I\subset D$.

For the purpose of this paper, we only consider the case of strict or non-relaxed differential privacy, where $\d=0$. In this case, the requirement simplifies to
$$\Prob(X_{i}=k)\le e^\e \Prob(X_{j}=k),$$
for all $i,j,k\in\{1, \dots, n\}$.

If we let $A\in\R^{n\times n}$ be given by,
$$a_{ij}= \Prob(X_{i}=j)$$
then $A$ defines a valid $\e$-differential privacy mechanism if and only if the following conditions hold:
\begin{subequations}\label{eq:dp0}
\begin{align}
\sum_{j} a_{ij}&=1, & &1\le i\le n,\label{eq:dp1}\\
a_{ij}&\ge 0, & &1\le i,j\le n,\label{eq:dp2}\\
a_{ik}&\le e^\e a_{jk}, & &1\le i,j,k\le n.\label{eq:dp3}
\end{align}
\end{subequations}


We now define the $\e$-differential privacy polytope, comprised of all matrices satisfying the above constraints.

\begin{definition}[Differential Privacy Polytope]\label{df:dp}
Fix $n\in\N$ and $\e\ge0$. The $\e$-differential privacy polytope, $\D\subset\R^{n\times n}$, is defined as follows:
\begin{align}\label{eq:sm}
\D=\left\{A\in\R^{n\times n}:\begin{array}{ll}\sum_j a_{ij}=1, & \forall\;1\le i\le n,\\
a_{ij}\ge 0, & \forall\;1\le i,j\le n,\\
a_{ij}\le e^\e a_{kj}, & \forall\;1\le i,j,k\le n.\end{array}\right\}.
\end{align}

The non-negativity and stochastic constraints ensure $\D$ is bounded. Therefore it is a polytope.

\end{definition}

\textbf{Note:} As the constraint $a_{ij}\le e^\e a_{kj}$ must hold for all $i,j,k$, we require $e^{-\e}a_{kj}\le a_{ij}\le e^\e a_{kj}$ for each $i,j,k$. Equivalently, $\max_i a_{ij}\le e^\e\min_i a_{ij}$ for all $j$.

\textbf{Remark:} If $\e=0$, then for $A$ to be in $\D$, we require that $a_{ij}=a_{kj}$ for all $i,j,k$.

Using the Hilbert Schmidt inner product 
$$\langle X,Y\rangle=\trace(X^TY),$$
together with the matrices $e_ie_j^T$, $e_i\mathbf{1}^T$, $e_ie_j^T = e^\e e_ke_j^T$, it is not a difficult exercise to represent the constraints defining $\D$ in the form given in Definition \ref{df:polyt}.

\section{Preliminary results}\label{sc:elem}



In this section, we present several preliminary results on the structure of the set $\D$ and its extreme points.  We first note that the nonnegativity constraint in the definition of $\D$ is redundant in the case where $\e > 0$.  

\begin{lemma}
Fix $\e>0$. Let $v\in\R^n$ satisfy $v_i\le e^\e v_j$ for all $i,j$. Then $v\ge 0$.
\end{lemma}

\begin{proof}
Let $v_i<0$ for some $i$. Then, for each $j$, we have:
\begin{align*}
& & e^{-\e} v_i&\le v_j\le e^\e v_i\\
&\Rightarrow &e^{-\e}v_i&\le e^\e v_i\\
&\Rightarrow &e^{-\e}&\ge e^\e\\
&\Rightarrow & \e&\le0
\end{align*}

By hypothesis $\e>0$. Hence, we must have $v_i\ge 0$ for each $i$.
\end{proof}

Our next lemma notes that if the differential privacy constraint is tight on two elements in a column, that those two elements must be the minimum and maximum entries of that column.

\begin{lemma}\label{lem:colmaxmin}
Let $v\in\R^n$ be a vector with $v_i\le e^\e v_j$ for all $1\le i,j\le n$. Suppose there exists at least one pair $i,j$ where $v_i=e^\e v_j$. Then $\min_k v_k=v_j$ and $\max_k v_k=v_i$.
\end{lemma}

\begin{proof}
Suppose there exists $v_l$ such that $v_l>v_i$. Then, $v_l>e^\e v_j$, contradicting the differential privacy constraints. Similarly, if $v_l<v_j$, then $e^\e v_l<v_i$. The result follows.
\end{proof}

Several of our results will relate the extreme points $A$ of $\D$ to the non-zero columns in $A$.  With this in mind, we formally define 
$$\gamma(A)=\{i\in\{1,\dots,n\}:A^{(i)} \neq 0\}.$$
So that $\gamma(A)$ consists of the indices of the non-zero columns of $A$ and $1\le|\gamma(A)|\le n$ gives the number of non-zero columns in $A$. 


Our next result concerns the rank of the extreme points of $\D$; first we note the simple observation that $\rank(A)\le|\gamma(A)|$ for all $A$.

\begin{theorem}\label{th:rank}
Let $A\in\ext(\D)$. Then
$$\rank(A)=|\gamma(A)|.$$
\end{theorem}

	\begin{proof}
	Suppose $A\in\ext(\D)$. As noted before, $\rank(A)\le|\gamma(A)|$. If $A$ has only one non-zero column, then clearly $\rank(A)=1=|\gamma(A)|$.
	
	Let $|\gamma(A)|>1$ and suppose $\rank(A)<|\gamma(A)|$. Then there exists $\eta\in\R^n$, $\eta\ne0$ and $\eta_i=0$ for all $i\notin\gamma(A)$ (i.e. whenever $A^{(i)}=0$), such that $\sum_i\eta_i A^{(i)}=0$.
	
	Let $B=A\diag(\eta)$. By construction, $B\mathbf{1}=0$.
	
	Consider $C = A-\Delta B$, $D = A+\Delta B$, where $0<\Delta<\frac{1}{\max_i |\eta_i|}$. Then,
	\begin{enumerate}
	\item $C$ and $D$ are stochastic, as $A$ is stochastic and $B \mathbf{1} = 0$;
	\item since $a_{ij}\le e^\e a_{kj}$ and $(1\pm\Delta\eta_j)>0$ for all $i,j,k$, it follows that $c_{ij}\le e^\e c_{kj}$, $d_{ij}\le e^\e d_{kj}$; and
	\item $C, D\ge0$.
	\end{enumerate}
	Hence, $C$ and $D$ are in $\D$ and $C \neq D$ as $B \neq 0$.
	
	However, $\frac{1}{2}\left(C+D \right)=A$, so $A\notin\ext(\D)$, a contradiction. Therefore, for every $A\in\ext(\D)$, $\rank(A)=|\gamma(A)|$.
	\end{proof}

We shall often make implicit use of the following simple corollary to the above result; essentially it states that for an extreme point $A$ with at least 2 non-zero columns, none of these columns can have all their entries equal.

\begin{corollary}
\label{cor:constantcols} Let $A \in \ext(\D)$ satisfy $|\gamma(A)| \geq 2$.  Then there is no $i \in \gamma(A)$, $k \in \mathbb{R}$ with $A^{(i)} = k \mathbf{1}$.  
\end{corollary}
\begin{proof}
Suppose that there is some $k \in \mathbb{R}$, $i_0 \in \gamma(A)$ such that $A^{(i_0)} = k \mathbf{1}$.  Clearly, $k \neq 0$ as $i_0 \in \gamma(A)$ and $k \neq 1$ as $|\gamma(A)| \geq 2$.  As $A$ is stochastic, 
$$\sum_{i \in \gamma(A)} A^{(i)} = \mathbf{1}$$
which implies that
$$(1-\frac{1}{k}) A^{(i_0)} + \sum_{i \in \gamma(A), i \neq i_0} A^{(i)} = 0.$$
This implies that $\rank(A) < |\gamma(A)|$, contradicting Theorem \ref{th:rank}.  
\end{proof}

It is clear from the definition that $\D$ is closed under row/column permutations. Our next result notes that this same invariance property also holds for extreme points.

\begin{theorem}\label{th:perm}
Let $A\in\D$ and let $P_1, P_2\in\{0,1\}^{n\times n}$ be permutation matrices. Then $P_1AP_2\in\D$. Furthermore, $A\in\ext(\D)$ if and only if $P_1AP_2\in\ext(\D)$.
\end{theorem}

\begin{proof}
From Definition~\ref{df:dp}, it clearly follows that $P_1AP_2\in\D$ if $A\in\D$, since permuting a matrix only changes the order of rows and columns, but elements in the same row/column will remain in a common row/column.

Now, let $A\in\ext(\D)$ and suppose that $P_1AP_2\notin\ext(\D)$ for some permutation matrices $P_1$ and $P_2$. Then, there exist $B\neq C\in\D$ such that,
$$P_1AP_2=\frac{1}{2}(B+C).$$
However, since $P^{-1}=P^T$ for any permutation matrix $P$, we have
$$A=\frac{1}{2}(P_1^TBP_2^T+P_1^TCP_2^T),$$
where $P_1^TBP_2^T\neq P_1^TCP_2^T$ since $B\neq C$. This is a contradiction since $A\in\ext(\D)$, hence $P_1AP_2\in\ext(\D)$ also. Thus,
$$A\in\ext(\D)\Rightarrow P_1AP_2\in\ext(\D),$$
and
$$ P_1AP_2\in\ext(\D)\Rightarrow P_1^TP_1AP_2P_2^T=A\in\ext(\D),$$
hence $A\in\ext(\D)$ if and only if $P_1 AP_2\in\ext(\D)$ for any permutation matrices $P_1$ and $P_2$.
\end{proof}

\subsection{Tight constraints}

We now examine the implications of Theorem~\ref{th:barvinok} for the extreme points of $\D$.  We first note a simple fact concerning the number of linearly independent differential privacy constraints that can be tight on an element of $\D$.  

In the next result, we use $\C^{dp}_j$ to denote the set of all tight differential privacy constraints acting on the $j$th column of a matrix $A$.  Formally, given $A$, this consists of all constraints such that $a_{ij} - e^\e a_{kj} = 0$ where $1 \le i, k \le n$.  

\begin{theorem}\label{th:linindconst}
Let $A\in\D$ be given.  Then, $\dim(\spn(\C^{dp}_j))=n$ if and only if $a_{ij}=0$ for each $i\in\{1,\dots,n\}$.
\end{theorem}

\begin{proof}
If we make the obvious identification of the $j$th column of $A$ with a column vector, $A^{(j)}$ in $\mathbb{R}^n$, then each constraint in $\C^{dp}_j$ can be identified with a vector of the form $(0, \ldots, 1, 0 , \ldots, -e^\e, 0, \ldots, 0)^T$ where the $1$ occurs in the $i$th position and $e^\e$ occurs in the $k$th position.  If $\dim(\spn(\C^{dp}_j))=n$, there are $n$ linearly independent vectors $v_1, \ldots, v_n$ such that $v_i^T A^{(j)} = 0$ for $1 \leq i \leq n$ so it follows trivially that $A^{(j)} = 0$. 

For the converse, it is enough to note that $A^{(j)} = 0$ implies that every differential privacy constraint acting on the $j$th column is tight and that there are $n$ linearly independent such constraints.  To see this consider the matrix $T$ with: $t_{ii} = 1$ for $1 \leq i \leq n$; $t_{i+1, i} = -e^\e$ for $1\leq i < n$; $t_{1n} =-e^\e$; $t_{jk} = 0$ otherwise.  It can readily be verified that $T$ is non-singular.  
\end{proof}

\textbf{Remark:} A direct consequence of Theorem~\ref{th:linindconst} is that $\dim(\spn(\C^{dp}_j))\le n-1$ for any $j\in\gamma(A)$. 

Our later characterisations of the extreme points of $\D$ shall rely on the following concept of \emph{loose entries}.

\begin{definition}[Loose entries of a matrix]
Given $A \in \D$, define
$$\lambda(A)=\Bigl\{(i,j):a_{ij}\notin\bigl\{e^\e\min_k a_{kj}, e^{-\e}\max_k a_{kj}\bigr\}\Bigr\}.$$

For a matrix $A\in\D$, we say the entry $a_{ij}$ is \textbf{loose} if $(i,j)\in\lambda(A)$.
\end{definition}

It follows from Lemma \ref{lem:colmaxmin} that for any $(i,j)$ there exists a $k$ such that $a_{ij}=e^{\pm\e}a_{kj}$ if and only if $(i,j)\notin\lambda(A)$.  

\begin{example}
Let $\e=ln(2)$ and
$$A=\frac{1}{7}\left(\begin{array}{ccc} 4 & 1 & 2 \\ 3 & 2 & 2 \\ 2 & 1 & 4\end{array}\right).$$
Then $\lambda(A)=\{(2,1)\}$, since $3\notin\{4,2\}$.
\end{example}

Our next result bounds the number of loose entries of an extreme point in terms of the number of non-zero columns.

\begin{theorem}\label{th:maxmmin}
Let $A\in\ext(\D)$ with $|\gamma(A)|\geq 2$.  Then,
$$|\lambda(A)|\le n-|\gamma(A)|.$$
\end{theorem}

	\begin{proof}
	Let  $A\in\ext(\D)$ and consider the following sets of constraints active on $A$.  We define $$\C^{dp} = \bigcup_{j \in \gamma(A)} \C^{dp}_j$$ to be the set of tight differential privacy constraints acting on the columns in $\gamma(A)$.  Note the following readily verifiable facts:
	\begin{itemize}
	\item[(i)] for $j \notin \gamma(A)$, every differential privacy constraint acting on column $j$ is tight;
	\item[(ii)] the $n$ stochastic constraints are tight;
	\item[(iii)] as $|\gamma(A)| \geq 2$, no non-zero column of $A$ is of the form $k \mathbf{1}$ where $k \in \mathbb{R}$. 
	\end{itemize}
	It follows from (ii) and Theorem \ref{th:barvinok} that the number of tight, linearly independent differential privacy constraints on $A$ must be $n^2 -n$.  Furthermore, Theorem \ref{th:linindconst} implies that there are $n$ linearly independent differential privacy constraints active on each of the $n - |\gamma(A)|$ zero columns of $A$.  It is not difficult to see that constraints acting on different columns must be linearly independent and hence there are a total of $(n-|\gamma(A)|)n$ linearly independent tight differentially private constraints arising from the zero columns of $A$.  Putting all of this together, we see that there must be $$n^2 - n - (n-|\gamma(A)|)n = n| \gamma(A)| - n$$ tight differential privacy constraints acting on the non-zero columns of $A$.  Formally:
\begin{equation}
\label{eq:OMcol1}
|\C^{dp}| \geq n| \gamma(A)| - n.
\end{equation}

From point (iii) above there are no non-zero columns in which all entries are constant; it follows that for each $j\in\gamma(A)$,
	$$|\{i:(i, j)\notin \lambda(A) \}|\ge|\C^{dp}_j|+1.$$
If we let $l_j$ denote the number of loose entries in column $j$, the previous inequality can be rewritten as 
$$|\C^{dp}_j| \le n- l_j -1.$$
Combining this with \eqref{eq:OMcol1} we see that
\begin{eqnarray*}
n| \gamma(A)| - n &\leq& \sum_{j \in \gamma(A)} |\C^{dp}_j| \\
&\le& \sum_{j \in \gamma(A)} n- l_j -1 \\
&=& n |\gamma(A)| - |\lambda(A)| - |\gamma(A)|.
\end{eqnarray*} 
A simple rearrangement now shows that
$$|\lambda(A)| \le n - |\gamma(A)|$$ as claimed. 
\end{proof} 

\textbf{Note:} When $|\gamma(A)|=1$, $|\lambda(A)|=n$.

To conclude this sub-section, we take a look at the following result for later use, which states that at most one loose entry can appear in any row of an extreme point.

\begin{lemma}\label{lm:samerow}
Let $A\in\ext(\D)$. No row of $A$ has more than one loose entry (i.e. there exist no two distinct pairs $(i_1,j_1),(i_1,j_2)\in \lambda(A)$ with $j_1\neq j_2$).
\end{lemma}

\begin{proof}
Let $A\in\ext(\D)$, and assume without loss of generality that $(1,1), (1,2)\in \lambda(A)$. Let
\begin{equation*}
\Delta=\min\left\{\max_i a_{i1}-a_{11}, a_{11}-\min_i a_{i1}, \max_i a_{i2}-a_{12}, a_{12}-\min_i a_{i2}\right\}.
\end{equation*}
Hence, $A\pm\Delta(E_{11}- E_{12})\in\D$.

However, $A=\frac{1}{2}((A+\Delta E_{11}-\Delta E_{12})+(A-\Delta E_{11}+\Delta E_{12}))$, hence, $A\notin\ext(\D)$, a contradiction and so the result follows.
\end{proof}

Finally, for this section we present a number of other results that will add further insight to the behaviour and structure of $\D$ and its extreme points.  The next piece of notation will prove useful later.  

For $A\in\D$, we define the vector $m'\in\R^n$ where $m'_j=\frac{1}{\min_i a_{ij}}$ for any $j\in\gamma(A)$ and $m'_j=0$ otherwise.  We then denote by $\tilde{A}$ the matrix given by:
\begin{equation}\label{eq:tildeA}
\tilde{A} =A\diag(m').
\end{equation}

Then, for any $A\in\D$, $\tilde{a}_{ij}\in[1,e^\e]$ for any $j\in\gamma(A)$, and $\tilde{a}_{ij}=0$ otherwise.

Hence, $$\tilde{A}\diag_{1\le j\le n}\left(\min_i a_{ij}\right)=A.$$

\textbf{Note:} $\gamma(A)=\gamma(\tilde{A})$ and $\lambda(A)=\lambda(\tilde{A})$.

We now show that for any extreme point $A$, $\tilde{A}$ cannot have a row with equal non-zero values.

\begin{lemma}\label{lm:nonconstrow}
Let $A\in\ext(\D)$ with $|\gamma(A)|>1$. Then for each row $i$, there exist two non-zero columns $j, k\in\gamma(A)$ such that $\tilde{a}_{ij} \neq \tilde{a}_{ik}$.
\end{lemma}

\begin{proof}
We prove this by contradiction. Firstly, suppose there exists a row $i$ such that $\tilde{a}_{ij}=\tilde{a}_{ik}$ for all $j,k\in\gamma(A)$. By Lemma~\ref{lm:samerow}, each row cannot have more than one loose element, therefore either $\tilde{a}_{ij}=1$ or $\tilde{a}_{ij}=e^\e$.

Let $m\in\R^n$ be defined by $m_j=\min_i a_{ij}$. Then $A=\tilde{A}\diag(m)$.

Suppose $\tilde{a}_{ij}=1$, hence $\sum_{k\in\gamma(A)}m_k=1$. By Theorem~\ref{th:maxmmin}, each column $j$ has at least one pair $(i,k)$ such that $a_{ij}=e^\e a_{kj}$, hence there exists a row $i^*$ such that $\tilde{a}_{i^*j}=e^\e$. However, $\tilde{a}_{i^*k} \ge1$ for every $k\in\gamma(A)$, so $\sum_{k\in\gamma(A)} \tilde{a}_{i^*k}m_k>1$, contradicting the stochasticity of $A$.

A similar argument holds for $\tilde{a}_{ij}=e^\e$. The result follows.
\end{proof}

%
%
%

\section{Extreme points for fixed values of $|\gamma(A)|$}\label{sc:main}

In this section, we characterise extreme points with a specified number of non-zero columns. We note that extreme points with one and two non-zero columns are limited to a specific form, while Section~\ref{sc:mainmainres} deals with extreme points with any number of non-zero columns.

\subsection{Extreme points with one column non-zero}

The first case to consider is that of a single non-zero column in the matrix. Due to the stochastic constraints, there are only $n$ such matrices, and as Theorem~\ref{th:thm1} below states, each one of these matrices is an extreme point.

\begin{theorem}[$|\gamma(A)|=1$]\label{th:thm1}
Let $E_i\in\R^{n\times n}$ be given by $E_i = \mathbf{1}e_i^T$ for $1 \leq i \leq n$ and define the set $\tilde{\D}'$ as:
$$\tilde{\D}'=\left\{E_1, \dots, E_n\right\}.$$

Then $\tilde{\D}'\subseteq\ext(\D)$.

Furthermore, $A\in\ext(\D)$, $|\gamma(A)|=1$ implies that $A\in\tilde{\D}'$.
\end{theorem}

\begin{proof}
Suppose $E_i = \frac{1}{2}(B+C)$ for $B, C$ in $\D$.  As $B$, $C$ are both nonnegative, it follows immediately that all columns of $B$ and $C$ apart from the $i$th column are zero.  $B$ and $C$ are also both stochastic which immediately implies that $B = C = \mathbf{1}e_i^T$.  

Note that if $A\in\D$ with $|\gamma(A)|=1$, then $A=E_i$ for some $i$. Hence, if $A\in\ext(\D)$ with $|\gamma(A)|=1$, it follows that $A\in\tilde{\D}'$.
\end{proof}

The points in $\tilde{\D}'$ are extreme points in all cases, regardless of $\e$. Furthermore, in the trivial case of $\e=0$, the set $\tilde{\D}'$ are the only extreme points.

\begin{corollary}
Let $\e=0$. Then,
$$\ext(\D)=\tilde{\D}'.$$
\end{corollary}

\begin{proof}
Let $\e=0$. Then, for all $A\in\D$, we have $a_{kj}\le a_{ij}\le a_{kj}$, hence $a_{ij}=a_{kj}$ for all $i,j,k$, i.e.\ entries in the same column are equal.  It now follows immediately from Corollary \ref{cor:constantcols} that if $A$ is an extreme point, $\gamma(A) = 1$ and hence that $A \in \tilde{\D}'$ as claimed. 
\end{proof}

\subsection{Extreme points with two columns non-zero}

Next, we consider the case of two non-zero columns. Although Theorem~\ref{th:maxmmin} allows for many loose entries to occur in these extreme points, Theorem~\ref{th:2nonzcols} below states that no loose entries are possible.

\begin{theorem}[$|\gamma(A)|=2$]\label{th:2nonzcols}
Let $A\in\ext(\D)$ where $|\gamma(A)|=2$. Then $A$ has no loose entries.
\end{theorem}

	\begin{proof}\begin{subequations}
	Without loss of generality, assume that $\gamma(A)=\{1,2\}$.  Define $m\in\R^n$ by $m_j=\min_i a_{ij}$ for $1 \leq j \leq n$ and define $\tilde{A}$ so that $A=\tilde{A}\diag(m)$.  Then $\tilde{a}_{ij}\in[1,e^\e]\cup\{0\}$ for $1 \leq i, j \leq n$. 
	
	By Theorem~\ref{th:maxmmin}, $|\lambda(A)|\le n-2$, so there exist at least two rows with no loose entries.  Let row $k$ be one of these rows. Then $\tilde{a}_{k1}, \tilde{a}_{k2}\in\{1,e^\e\}$, but by Lemma~\ref{lm:nonconstrow}, $\tilde{a}_{k1}\neq \tilde{a}_{k2}$.  We can assume that $\tilde{a}_{k1}=e^\e$ and $\tilde{a}_{k2}=1$ (otherwise just swap columns 1 and 2). As $A$ is stochastic, 
	\begin{equation}\label{eq:th11a}
	m_1 e^\e+m_2=1.
	\end{equation}
	
	By Lemma~\ref{lm:samerow}, for all rows $j$, at least one of $\tilde{a}_{j1}$, $\tilde{a}_{j2}$ must be in $\{1,e^\e\}$.  Moreover, in order to satisfy (\ref{eq:th11a}), $\tilde{a}_{j1}=e^\e$ if and only if $\tilde{a}_{j2}=1$.
	
	Suppose therefore that there exists a row $j$ where $\tilde{a}_{j1}\in(1,e^\e)$ corresponding to a loose entry in $A$.  It follows from \eqref{eq:th11a} that $\tilde{a}_{j2}=e^\e$. Hence
	\begin{equation}\label{eq:th11b}
	\begin{split}
	1&=m_1\tilde{a}_{j1}+m_2e^\e\\
	&> m_1+m_2e^\e.
	\end{split}
	\end{equation}
	
	It follows from Corollary \ref{cor:constantcols} that there is some $j^*$ such that $\tilde{a}_{j^*1}=1$, implying
	\begin{equation*}\begin{split}
	1&=m_1+m_2\tilde{a}_{j^*2}\\
	&\le m_1+m_2e^\e,
	\end{split}\end{equation*}
	contradicting (\ref{eq:th11b}). Therefore there are no loose entries in the first column.
	
	Now suppose there exists a row $j$ where $\tilde{a}_{j2}\in(1,e^\e)$.  As above, it follows that $\tilde{a}_{j1}=1$. Hence,
	\begin{equation}\label{eq:th11d}
	\begin{split}
	1&=m_1+m_2\tilde{a}_{j2}\\
	&<m_1+m_2e^\e.
	\end{split}
	\end{equation}
	
	As before, it follows from Corollary \ref{cor:constantcols} that there is some $j^*$ such that $\tilde{a}_{j^*2}=e^\e$, hence,
	\begin{equation*}\begin{split}
	1&=m_1\tilde{a}_{j^*1}+m_2e^\e\\
	&\ge m_1+m_2e^\e,
	\end{split}\end{equation*}
	contradicting (\ref{eq:th11d}). Therefore there are no loose entries in the second column.
	
	Hence $|\lambda(A)|=0$.
	\end{subequations}\end{proof}
	
Using this result along with Lemma~\ref{lm:nonconstrow}, we can describe the two non-zero columns.

\begin{corollary}\label{cr:2nzcols}
Let $A\in\ext(\D)$ with $|\gamma(A)|=2$. Let $\gamma(A)=\{j,k\}$ and $\tilde{A}$ be given by \eqref{eq:tildeA}. Then, for every $1\le i\le n$, we have
\begin{align}\label{eq:gamma2}
(\tilde{a}_{ij}, \tilde{a}_{ik})\in\{(1, e^\e), (e^\e, 1)\}.
\end{align}
\end{corollary}

	\begin{proof}
	By Theorem~\ref{th:2nonzcols}, $\tilde{a}_{ij}\in\{1,e^\e\}$ and $\tilde{a}_{ik}\in\{1, e^\e\}$ for each $1\le i\le n$.
	
	By Lemma~\ref{lm:nonconstrow}, we must have $\tilde{a}_{ij}\neq\tilde{a}_{ik}$ for each $i$. So, $\tilde{a}_{ij}=e^\e$ if and only if $\tilde{a}_{ik}=1$, and $\tilde{a}_{ij}=1$ if and only if $\tilde{a}_{ik}=e^\e$.
	\end{proof}

The follow example illustrates the consequence of Corollary~\ref{cr:2nzcols}.

\begin{example}
Every extreme point $A\in\ext(\D)$ with $|\gamma(A)|=2$ must be of the form shown in (\ref{eq:gamma2}), and furthermore both non-zero columns of $\tilde{A}$ must contain at least one 1 and one $e^\e$.

Let $n=4$ and $A\in\ext(\D)$ with $|\gamma(A)|=2$. One example of such an $A$ is as follows:
$$A=\frac{1}{1+e^\e}\left(\begin{array}{cccc}1&0&e^\e&0\\1&0&e^\e&0\\e^\e&0&1&0\\1&0&e^\e&0\end{array}\right)\in\ext(\D).$$
\end{example}

\subsection{Extreme points with every element constrained}\label{sc:mainmainres}

The next definition is necessary before we can state Theorem \ref{th:thm2} which is the main result of the paper.   

\begin{definition}
Let $\tilde{\D}\subset\D$ be defined as follows:
\begin{equation}\label{eq:dtilde}
\tilde{\D}=\{A\in\D\mid\rank(A)=|\gamma(A)|, \lambda(A)=\emptyset\}.
\end{equation}
\end{definition}

The set $\tilde{\D}$ contains matrices with between 2 and $n$ non-zero columns, which satisfy the rank condition of Theorem \ref{th:rank} and have no loose entries (i.e.\ $\tilde{a}_{ij}\in\{0,1,e^\e\}$ for each $i,j$). We now show that every one of these matrices is an extreme point of $\D$.

\begin{theorem}\label{th:thm2}
Let $\e>0$. Then,
$$\tilde{\D}\subset\ext(\D).$$
\end{theorem}

	\begin{proof}\begin{subequations}
	Let $A\in\tilde{\D}$ and let $B,C\in\D$ where $\frac{1}{2}(B+C)=A$. Define $m_j=\min_i a_{ij}$ for each $j\in\{1,\dots, n\}$ (Note that $m_j=0$ for each $j\notin\gamma(A)$, and $a_{ij}\in\{m_j, e^\e m_j\}$ for each $i,j$ since $\lambda(A)=\emptyset$).
	
	Let $\Delta_j=\frac{1}{2}\max_{i}|b_{ij} - c_{ij}|$ for each $j\in\gamma(A)$.  As $B$ and $C$ are nonnegative, it is not hard to see that:
	\begin{align}\label{eq:pt0}
	\Delta_j=0,\quad\forall\;j\notin\gamma(A).
	\end{align}
	
	We shall show that the same conclusion must also hold for $j \in \gamma(A)$.   To this end, let $j^*\in\gamma(A)$ be given where $\Delta_{j^*}>0$.  Assume without loss of generality that $b_{i_1j^*}=a_{i_1j^*}+\Delta_{j^*}$ for some $i_1$ (if not, swap $B$ and $C$). 

We claim that $a_{i_1j^*}\neq m_{j^*}$.  Suppose otherwise.  Then there exists $i_2$ where $a_{i_2j^*}=e^\e m_{j^*}$. However, since $\frac{1}{2}(B+C)=A$, we have $c_{i_1j^*}=2a_{i_1j^*}-b_{i_1j^*}=a_{i_1j^*}-\Delta_{j^*}$, and since $C\in\D$, we have
	\begin{align*}
	c_{i_2j^*}&\le e^\e c_{i_1j^*}\\&=e^\e a_{i_1j^*}-e^\e\Delta_{j^*}\\&=a_{i_2j^*}-e^\e\Delta_{j^*}
	\end{align*}
	By the definition of $\Delta_{j^*}$, we must have $c_{i_2j^*}\ge a_{i_2j^*}-\Delta_{j^*}$. Hence it would follow that $\Delta_{j^*}\ge e^\e\Delta_{j^*}$, a contradiction since $\e>0$.  Thus, $a_{i_1j^*}=e^\e m_{j^*}$ as claimed (i.e. the max change occurs on the max element of the column).
	
	We now know that $b_{i_1j^*}=e^\e m_{j^*}+\Delta_{j^*}$. Let $$I_{j^*}=\{i:a_{ij^*}=m_{j^*}\}.$$ Then for every $i\in I_{j^*}$, since $B\in\D$, we get $e^\e m_{j^*}+\Delta_{j^*}= b_{i_1j^*} \le e^\e b_{ij^*}$, hence
	\begin{align}\label{eq:pt2a}
	b_{ij^*}\ge m_{j^*}+e^{-\e}\Delta_{j^*}.
	\end{align}
	Also, for every $i\in I_{j^*}$, since $C\in\D$,
	\begin{align*}c_{i_1j^*}&=e^\e m_{j^*}-\Delta_{j^*}\\&\le e^\e c_{ij^*}\\&=e^e(2a_{ij^*}-b_{ij^*})\\&=2e^\e m_{j^*}-e^\e b_{ij^*},\end{align*}
	hence $e^\e m_{j^*}-\Delta_{j^*}\le2e^\e m_{j^*}-e^\e b_{ij^*}$, or rewriting,
	\begin{align}\label{eq:pt2b}
	b_{ij^*}\le m_{j^*}+e^{-\e}\Delta_{j^*}.
	\end{align}
	
	Hence, from (\ref{eq:pt2a}) and (\ref{eq:pt2b}),
	\begin{equation}\begin{split}\label{eq:prt2}
	b_{ij^*}&=m_{j^*}+e^{-\e}\Delta_{j^*}\\
	&=a_{ij^*}+e^{-\e}\Delta_{j^*},
	\end{split}\end{equation}
	for every $i\in I_{j^*}$.
	
	It follows readily that for every $i \in I_{j^*}$, $c_{ij^*} = m_{j^*}-e^{-\e}\Delta_{j^*}$.  
	
	We next consider indices $i \notin I_{j^*}$.   Choose some $i_2\in I_{j^*}$.  For all $i\notin I_{j^*}$, $a_{ij^*}=e^\e m_{j^*}$, then
	\begin{equation}\label{eq:pt3a}
	b_{ij^*}\le e^\e b_{i_2j^*}=e^\e m_{j^*}+\Delta_{j^*},
	\end{equation}
	and
	\begin{align*}c_{ij^*}&=2a_{ij^*}-b_{ij^*}\\&=2e^\e m_{j^*}-b_{ij^*}\\&\le e^\e c_{i_2j^*}\\&=e^\e m_{j^*}-\Delta_{j^*},\end{align*}
	which can be rewritten as
	\begin{equation}\label{eq:pt3b}
	b_{ij^*}\ge e^\e m_{j^*}+\Delta_{j^*}.
	\end{equation}
	
	Hence, from (\ref{eq:pt3a}) and (\ref{eq:pt3b}),
	\begin{equation}\begin{split}\label{eq:prt3}
	b_{ij^*}&=e^\e m_{j^*}+\Delta_{j^*}\\
	&=a_{ij^*}+\Delta_{j^*},
	\end{split}\end{equation}
	for all $i\notin I_{j^*}$.
	
	Putting everything together, it follows from (\ref{eq:pt0}), (\ref{eq:prt2}) and (\ref{eq:prt3}),
	$$b_{ij} = \begin{cases} 0, & j\notin\gamma(A), \\ m_j+e^{-\e}g_j\Delta_j, & j\in\gamma(A), i\in I_j \\ e^\e m_j+g_j\Delta_j, & j\in\gamma(A), i\notin I_j \end{cases}$$
	where $g_j\in\{-1,1\}$, for all $j\in\gamma(A)$.
	
	Similarly, since $B+C=2A$,
	$$c_{ij} = \begin{cases} 0, & j\notin\gamma(A), \\ m_j-e^{-\e}g_j\Delta_j, & j\in\gamma(A), i\in I_j \\ e^\e m_j-g_j\Delta_j, & j\in\gamma(A), i\notin I_j. \end{cases}$$
	
	Rewriting in terms of $\tilde{A}$ (given by \eqref{eq:tildeA}), $b_{ij}=a_{ij}+g_j e^{-\e} \frac{a_{ij}}{m_j}\Delta_j=a_{ij}+g_je^{-\e}\tilde{a}_{ij}\Delta_j$ and $c_{ij}=a_{ij}-g_je^{-\e}\tilde{a}_{ij}\Delta_j$ for all $i,j$.
	
	Hence,
	\begin{align*}
	B&=A+e^{-\e}\tilde{A}\diag_{1\le j\le n}(g_j \Delta_j)\\
	C&=A-e^{-\e}\tilde{A}\diag_{1\le j\le n}(g_j \Delta_j).
	\end{align*}
	\end{subequations}
	Since $A,B$ are stochastic, we require
	$$e^{-\e}\tilde{A}\diag_{1\le j\le n}(g_j \Delta_j)\mathbf{1}=0.$$
	
	This equation defines a linear relationship between the columns of $\tilde{A}$.  Moreover, we know that $\Delta_j = 0$ for $j \notin \gamma(A)$.  If $\Delta_j \neq 0$ for any $j \in \gamma(A)$, it would imply that the non-zero columns of $\tilde{A}$ and hence those of $A$ are linearly dependent, contradicting the assumption that $\textrm{rank}(A) = |\gamma(A)|$.  It follows that $\Delta_j = 0$ for all $j$ and hence that $B = C = A$.  This completes the proof.  
\end{proof}

Furthermore, the set $\tilde{\D}$ contains all extreme points of $\D$ which have no loose entries.

\begin{corollary}\label{cr:thm2}
Let $A\in\D$ with $\lambda(A)=\emptyset$. Then, $A\in\ext(\D)$ if and only if $A\in\tilde{\D}$.
\end{corollary}

	\begin{proof}
	``$\Rightarrow$": Let $A\in\ext(\D)$ with $\lambda(A)=\emptyset$. By Theorem~\ref{th:rank}, $\rank(A)=|\gamma(A)|$, hence $A\in\tilde{\D}$.
	
	``$\Leftarrow$": $A\in\tilde{\D}\Rightarrow A\in\ext(\D)$ by Theorem~\ref{th:thm2}.
	\end{proof}

\subsection{Extreme points with all columns non-zero}

From an application point of view, it is entirely reasonable to only consider matrices (and the resulting response mechanism) with no zero columns.

Having a zero column in a matrix that defines a response mechanism means that the mechanism never releases a particular (or multiple) values as its output. In many circumstances, this feature will not be required of a mechanism.

Using Theorem~\ref{th:thm2}, we now present the following corollary, which gives a complete characterisation of extreme points without zero columns.

\begin{corollary}\label{cr:main}
Let $A\in\D$, with $|\gamma(A)|=n$. Then, $A\in\ext(\D)$ if and only if $A\in\tilde{\D}$

Equivalently,
$$\{A\in\ext(\D):|\gamma(A)|=n\}=\{A\in\tilde{\D}:|\gamma(A)|=n\}.$$
\end{corollary}

\begin{proof}
``$\Rightarrow$'': Let $A\in\ext(\D)$ have $n$ non-zero columns. Then, $\rank(A)=n$ by Theorem~\ref{th:rank} and $\lambda(A)=\emptyset$ by Theorem~\ref{th:maxmmin}.

``$\Leftarrow$'': Let $A\in\D$ such that $\rank(A)=n$ and $\lambda(A)=\emptyset$. Then $A\in\ext(\D)$ by Theorem~\ref{th:thm2}.
\end{proof}

We now have necessary and sufficient conditions for finding and determining extreme points with $n$ non-zero columns.

\section{Discussion}\label{sc:special}

We now take a brief look at a number of useful and interesting consequences of the results given in Sections~\ref{sc:elem} and \ref{sc:main}.

\begin{description}
\item[$\ext(\D)$ for small $n$:]  From Theorems~\ref{th:thm1} and \ref{th:thm2}, we know that $\tilde{\D}'\cup\tilde{\D}\subseteq\ext(\D)$; with the addition of Theorem~\ref{th:2nonzcols} we can make further observations for small $n$.
\begin{theorem}
Let $n\le3$, then
$$\ext(\D)=\tilde{\D}'\cup\tilde{\D}.$$
\end{theorem}
\item[Extreme points for $n=4$:] We therefore have a complete characterisation of all extreme points up to $n=3$.  While we lack a formal proof, extensive computer simulations suggest it is also true for $n=4$ leading to the following conjecture.\\
Let $n\le4$: then
$$\ext(\D)=\tilde{\D}'\cup\tilde{\D}.$$
\item[$\ext(\D)$ for $n\ge5$:] When $n = 5$, our previous results allow us to characterise all extreme points $A$ for which $|\gamma(A)|=1,2,5$.  However, when $|\gamma(A)|=4$, we can find extreme points with loose entries.

The following point $A\in\D$ can be shown to be an extreme point of $\D$ by using Theorem~\ref{th:barvinok}.
$$A=\frac{1}{3+2e^\e}\left(\begin{array}{ccccc} 1 & 1 & 2e^\e & 1 & 0 \\
e^\e & 1 & 2 & e^\e & 0 \\
e^\e & e^\e & 2 & 1 & 0 \\
1 & e^\e & 2 & e^\e & 0 \\
1 & 1 & 1+e^\e & e^\e & 0\end{array}\right).$$
Fitting with Theorem~\ref{th:maxmmin}, $A$ has only a single loose entry ($\lambda(A)=\{(5,3)\}$), while we also observe that $\rank(A)=4$, satisfying Theorem~\ref{th:rank}.

We therefore have $A\in\ext(\D)$, but $A\notin\tilde{\D}'\cup\tilde{\D}$. Hence, $\tilde{\D}'\cup\tilde{\D}\subset\ext(\D)$ in general.
\end{description}

\section{Conclusion}\label{sc:conc}

We have studied the differential privacy polytope of $n\times n$ matrices and described a suite of results characterising its extreme points. In particular, our results describe completely the extreme points of this polytope containing 1, 2 and $n$ non-zero columns.  The last fact is of particular practical significance as most implementations of differentially private mechanisms are likely to have no zero columns; this is because a zero column corresponds to a value of the dataset $D$ that is never released by the mechanism.  Future work could focus on characterising extreme points with other values of $|\gamma(A)|$; alternative directions for work include considering other convex geometric aspects of the polytope $\D$ such as the structure of its dual set for example.  

\section*{Acknowledgment}
The first named author was supported by the Science Foundation Ireland grant SFI/11/PI/1177.

\end{document}